\documentclass[11pt]{article}

\usepackage[OT2,OT1]{fontenc}
\newcommand\cyr
{
\renewcommand\rmdefault{wncyr}
\renewcommand\sfdefault{wncyss}
\renewcommand\encodingdefault{OT2}
\normalfont
\selectfont
}
\DeclareTextFontCommand{\textcyr}{\cyr}
\def\cprime{\char"7E }

\usepackage{tikz-cd}
\usepackage{amsmath}
\usepackage{amssymb}
\usepackage{amsthm}

\theoremstyle{plain}
\newtheorem{prop}{Proposition}[section]
\newtheorem{thm}[prop]{Theorem}
\newtheorem{lemma}[prop]{Lemma}
\newtheorem{defn}[prop]{Definition}

\usepackage{color}
\usepackage[top = 2.75cm, bottom = 2.50cm, left   = 2.00cm, right  = 2.00cm]{geometry}

\newcommand{\fcy}[1]{\mathcal{#1}}
\newcommand{\bb}[1]{\mathbb{#1}}

\newcommand{\bo}[1]{{\bf #1}}
\newcommand{\ra}{\rightarrow}

\newcommand{\ak}{\alpha}        
\newcommand{\bk}{\beta}         
\newcommand{\gk}{\gamma}        \newcommand{\Gk}{\Gamma}
\newcommand{\dk}{\delta}        
\newcommand{\ek}{\varepsilon}

\newcommand{\lk}{\lambda}

            \newcommand{\X}{\Xi}

          \newcommand{\Fk}{\Phi}
          
          \newcommand{\Yk}{\Psi}

\newcommand{\nas}{\bo{NAS}_\bullet}

\begin{document}

\title{Non-Anomalous Semigroups and Real Numbers}
\author{Damon Binder \\}
\maketitle
\begin{abstract}
Motivated by intuitive properties of physical quantities, the notion of a non-anomalous semigroup is formulated. These are totally ordered semigroups where there are no `infinitesimally close' elements. The real numbers are then defined as the terminal object in a closely related category. From this definition a field structure on $\bb R$ is derived, relating multiplication to morphisms between non-anomalous semigroups.
\end{abstract}

\section{Introduction}
In this paper, we give a new characterization of the reals: we define the category of pointed non-anomalous semigroups, and identify the reals as the terminal object here. This avoids attributing to $\bb R$ all but the barest semigroup and order structures. We then show that the other structures on the reals naturally follow from this definition. In particular, multiplications originates from the morphisms of the category.

Our motivation is to give a simple and well-motivated definition of the reals. Real numbers are of central importance in both mathematics and science, and so we would expect a mathematical characterization which is simple and elegant. Most crucially, we would hope there to be an intimate connection with our intuitive and philosophical notions of quantity.

There are two main traditional approaches to defining real numbers. The axiomatic approach defines $\bb R$ as the unique complete totally ordered field. This approach involves introducing three structures, addition, multiplication, and order, along with a large number of axioms (around fifteen). The other approach is the constructive approach, where first $\bb N$ is defined, and then from there $\bb Z$, $\bb Q$, and finally $\bb R$ are constructed.

Both approaches are complicated, and the connection to quantity opaque. In the axiomatic approach, the axioms are numerous and difficult to justify. The most pertinent problem is with multiplication. If we had a collection of weights, it is intuitive that they can be ordered, and that weights can be combined (``added"). Yet no clear meaning can be assigned to multiplying two weights. Units reflect this: adding kilograms gives us kilograms, yet multiplying gives us Mg$^2$. We cannot combine nor order kg with Mg$^2$. So though multiplying weights produces a real number, there is no canonical isomorphism between the original quantities and their product.

We remark that even in purely mathematical contexts, multiplication plays a secondary role. In the definition of both measure and metric spaces, order and addition are needed in the axioms. Yet multiplication is not required, so we can trivially to generalize these structures to any ordered group.

We begin by studying totally ordered semigroups. These are the most general objects we can consider where elements can be both added and ordered. A notion of infinitesimally close elements is formulated, originally due to Alimov \cite{11}. We introduce the term `non-anomalous' to describe semigroups lacking infinitesimally close elements. This generalizes the notion of an Archimedean group to semigroups.

In section 3 we define the reals (considered as an ordered semigroup under addition) as the `biggest' possible non-anomalous semigroup. More specifically, $\bb R$ is the terminal object in a category we call the category of pointed non-anomalous semigroups, $\nas$. In particular this means that every non-anomalous semigroup is a subgroup of the reals. Proving that $\nas$ has a terminal object is non-trivial, and will occupy the bulk of the section.

Our definition of the reals is best understood in light of H\"older's theorem \cite{12}. This theorem states that any Archimedean group can be embedded into the reals under addition. Although H\"older's theorem is a statement about the additive and order properties of the real numbers, every proof we are aware of relies on the multiplicative properties of $\bb R$. Our characterization is a companion H\"older's theorem in the opposite direction: we define the real numbers to be the `largest' possible non-anomalous semigroup. This definition minimizes the number of extraneous properties attributed to the reals.

The main result in section 4 is a description of $\bb R$. We show that our definition of $\bb R$ gives a dense and complete totally ordered group, and furthermore, any other dense and complete totally ordered group is isomorphic to $\bb R$. This connects our definition of $\bb R$ to more traditional approach, since by a result from Loonstra \cite{13}, $\bb R$ is the unique dense complete totally ordered group. 

Finally, section 5 relates the properties of $\nas$ to rings and fields. Multiplication is induced by homomorphisms of non-anomalous semigroups. We prove that $\bb R$ has a unique field structure, and furthermore, that $\bb R$ is the unique complete ordered field up to a unique isomorphism.

\section{Quantity and Ordered Semigroups}

\subsection{Axioms for Quantity}
Our first task is to justify the relation between quantities and totally ordered semigroups. To make the discussion concrete, consider a collection of weights along with a balance scale. Placing weights $X$ and $Y$ on either sides of the scale, we find that the weight $X$ always rises. This seems important, so we decide to introduce a symbol $<$ and write $X<Y$ if $X$ rises and $Y$ falls when both are placed on a scale. Obviously, if $X<Y$ we know that $Y<X$ does not hold. If neither $X<Y$ and $Y<X$, then the scale we have cannot distinguish the two weights, and so we decide to say that they are copies of the same weight, $X=Y.$

Comparing more weights, we notice a pattern; if $X<Y$ and $Y<Z$ we find that $X<Z$. So our weights are in fact totally ordered.

We then discover that we can glue weights together, treating them as a single weight. So given weights $X$ and $Y$, we write $X+Y$ to mean the weight gained by sticking $X$ and $Y$ together. We notice that the order we stick our weights together does not matter:
$$X+Y = Y+X, \ \ \ (X+Y)+Z = X+(Y+Z).$$
We also find that if $Y<Z$, then gluing a weight $X$ on to both these weights will preserve this fact
\begin{equation}\label{oa}Y<Z \implies Y+X<Z+X\end{equation}

By considering empirical properties of weights, we have discovered many facts about them. Pithily, we can say that our collection of weights forms a totally ordered commutative semigroup, with \eqref{oa} governing the interaction of the two structures.

We considered weights, but there are many other things can also be considered as totally ordered commutative semigroups. Starting with sticks, we can compare the length of sticks to find the longer stick, and we can lay sticks end to end to produce a new stick; these operations give a totally ordered commutative semigroup structure. Or we can think about the time required to complete tasks, or the money required to buy an item, or the probability that biased coins will all land on heads.

So to study quantities, we will begin with totally ordered semigroups.

\begin{defn} A \bo{totally ordered semigroup} (which we will abbreviate to TOS) is a set $S$ along with a binary relation $<$ and a law of composition so that for any $x,y,z\in S$
\begin{enumerate}
\item If $x<y$ and $y<z$ then $x<z$
\item Exactly one of the following holds: $x<y$, or $x = y$, or $x>y.$
\item $(xy)z = x(yz)$
\item If $x<y$, then $xz<yz$ and $zx<zy.$
\end{enumerate}
\end{defn}

For the moment we have dropped the requirement of commutativity; later we will show that this can be derived from other hypotheses. We will therefore use multiplicative notation throughout this section, and switch to additive notation only once we restrict to commutative objects.

We should also note that the fourth axiom implies that the semigroup is cancellative. If we had instead used $\leq$ instead, this would not be the case. Some authors use definitions which do not require cancellativity.

Given a TOS $S$, we can define the \emph{dual} TOS $\bar S$ to have the same group structure, but reversed inequalities. That is, we have a map $\dk_S: S\ra \bar S$ which is a group isomorphism, and has the property that if $x<y$ then $\dk_S(x)>\dk_S(y).$ Duality allows us to make new definitions and theorems from old ones, by reversing all the inequalities that appear.

\begin{defn} We say an element $x\in S$ is positive if $x^2>x$, and negative if $x^2<x$. We will use $S^+$ to denote the positive elements of $S$, and $S^-$ to denote the negative elements.
\end{defn}

We can see that negative elements are defined in a manner dually to positive elements. If our semigroup has an identity $e$, then from the cancellation law these definitions give the traditional definition of positivity, $x>e$.

\begin{prop}\label{P21} For $x$ in a TOS $S$, $x\in S^+$ is equivalent to $ax>a$ for every $a\in S.$
\end{prop}
\begin{proof} First assume that $x\in S^+$. Then $x^2>x$ and so $ax^2>ax$ for any $a\in S$. Canceling the $x$ on the left, $ax>a.$ Conversely, if $ax>a$ for every $a\in S$, then for the case that $x = a$, $x^2>x$ and hence $x$ is positive.
\end{proof}

The properties we expect of positive and negative elements follow from this proposition in a straightforward manner. For instance, the product of positive elements is positive. If $y>x$ and $x$ is positive, then $y$ is positive. Finally, it implies that any non-positive and non-negative element must be an identity.

We end the section with a technical lemma which will be useful in the next section.

\begin{lemma}\label{L21} If $xy>yx$, then $x^ny^n>(xy)^n>(yx)^n>y^nx^n$ for every $n\in\bb N$.\end{lemma}
\begin{proof} When $n = 1$ the statement is trivially true. Assuming it is true for $k$, we then find using the first inequality
$$x^{n+1}y^{n+1}>x^nyx y^n>x^ny^2xy^{n-2} > ... > x^ny^nxy > (xy)^n xy = (xy)^{n+1}$$
and using the second inequality
$$(xy)^{k+1} = (xy)^k(xy)>(yk)^k(xy)>(yk)^k(yx) = (yx)^{k+1}.$$
The third inequality is dual to the first, and so
$$x^{n+1}y^{n+1}>(xy)^{n+1}>(yx)^{k+1}>y^{n+1}x^{n+1}$$
The lemma now follows by induction.\end{proof}

\subsection{Infinities and Infinitesimals}
So far our axioms are not quite strong enough to capture the important properties of numbers. In general totally ordered groups can exhibit very wild behavior. Chehata \cite{21} and Vinogradov \cite{22} independently constructed the same example of totally ordered semigroup which cannot be embedded into any group. Even totally ordered groups can be very complicated. For instance, every free group can be totally ordered. A proof of this fact, along with a detailed discussion of many other ordered groups and their applications to topology, can be found in \cite{23}.

The important property we are looking for is that there are no infinitely big or infinitesimally small quantities. In fact, all we require is that no two elements are infinitesimally close to each other. This is formalized by the notion of an anomalous pair, and is due to Alimov \cite{11}.

\begin{defn} Elements $x,y\in S$ with $x>y$ form an \bo{anomalous pair} if either
$$x^n<y^{n+1} \ \ \text{ or } \ \ y^n>x^{n+1} \ \ \text{ for every n}\in \bb N.$$
The former case implies that $x,y\in S^+$ and the latter implies that $x,y\in S^-$. If no pair in $S$ is anomalous, we shall call $S$ an \bo{non-anomalous} semigroup, or an NAS for short.
\end{defn}

Intuitively, an anomalous pair $x,y\in S^+$ with $x>y$ is a pair of elements where $x$ is infinitesimally larger than $y$, so that for any $n\in\bb N$,
$$y^{n+1}>x^n>y^n.$$

Given elements $x,y\in S^+$, we might consider $x$ to be infinitely larger than $y$ if for every $n\in \bb N$, $x>y^n.$ This means that no matter how many copies of $y$ we take, $x$ is still larger than the combination of $y$'s.
\begin{defn} A semigroup is \bo{Archimedean} if
\begin{enumerate}
\item For every $x,y\in S^+$ there exists an $n\in \bb N$ so that $x<y^n.$
\item For every $x,y\in S^-$ there exists an $n\in \bb N$ so that $x>y^n.$
\end{enumerate}
\end{defn}
The Archimedean property effectively requires that there are no `infinitely big' elements in the semigroup.

\begin{prop}[Alimov]\label{P22} Any non-anomalous semigroup is Archimedean.
\end{prop}
\begin{proof} Given an TOS $S$, let $x,y\in S^+$ be such that $y^n<x$ for every $n\in\bb N$. Applying Lemma \ref{L21}, we find that
$$x^n < y^nx^n < (xy)^n, \ \ \ (xy)^n < x^ny^n < x^{n+1} $$
and so $x$ and $xy$ form an anomalous pair. A dual argument holds for the case where $x,y\in S^-.$
\end{proof}

We are now in a position to prove our first major result, that any non-anomalous semigroup is commutative. From this perspective, the commutativity of addition in the real numbers is not an axiom, but a consequence of the fact that the real numbers are non-anomalous. The proof will require two technical lemmas, which we present first.

\begin{lemma}\label{L22} If $S$ is an Archimedean semigroup and $x,y\in S^+$ with $x>y$, then there exists an $n\in\bb N$ such that $y^{n+1}>x\geq y^n$.
\end{lemma}
\begin{proof} Since $S$ is Archimedean we know that there exists an $m\in\bb N$ with $x<y^m$. We also know that $x>y$, and hence there must exist a maximum $n\in\bb N$ satisfying $x\geq y^n$ but $y^{n+1}>x.$
\end{proof}

\begin{lemma}\label{L23} Let $S$ be non-anomalous with $x\in S^+$ and $y\in S^-$. There exists an $n\in\bb N$ such that $xy^n\in S^+.$
\end{lemma}
\begin{proof} If $xy\in S^+$ then the result follows trivially, so assume that $xy\in S^-$. Since $y$ and $xy$ are non-anomalous, there is a $m\in \bb N$ with
$$y^m<(xy)^{m+1}\implies xy^m<x(xy)^{m+1}.$$
Let us first assume that $xy>yx$, then applying Lemma \ref{L21},
$$xy^m<x(xy)^{m+1} < x^{m+2}y^{m+1}\implies x<x^{m+1}y.$$
Since $x$ is positive this implies that $x^{m+1}y$ is positive, and this completes the proof.
\end{proof}

\begin{thm}[Alimov]\label{T21} Any non-anomalous semigroup $S$ is commutative.
\end{thm}
\begin{proof} We will begin by showing that any two positive elements of $S$ must commute. This will be achieved through contradiction, assuming $x,y\in S^+$ do not commute and $x>y$. Without loss of generality we can take $xy>yx$.

Since $xy$ and $yx$ are not anomalous, there exists an $n\in\bb N$ with
$$(xy)^n>(yx)^{n+1}.$$

Using Lemma \ref{L21}, we find that
$$x^ny^n>(xy)^n > (yx)^{n+1} = (yx)^nyx > (yx)^ny^2.$$

Using Lemma \ref{L22} there exists $m\in\bb N$ so that $y^{m+1}>x^n\geq y^m,$ and so
$$y^{m+1}y^n>x^ny^n>y^nx^ny^2>y^ny^my^2.$$
But this then implies that
$$y>y^2,$$
which contradicts the fact that $y\in S^+.$

To show this suffices to prove the general case, assume $a,b\in S$ do not commute. There are three possibilities, of which, the possibility that both are positive has been ruled out. Instead if both $a$ and $b$ are negative, the dual of $S$ has noncommuting positive elements which is impossible. Finally, if only of $a$ is positive, applying Lemma \ref{L23} there exists a $k\in \bb N$ so that $c = a^kb$ is positive. From the cancellative law,
$$ac = a^k(ab) \neq a^k(ba) = ca,$$
and so $a$ and $c$ are noncommuting positive elements in $S$.\end{proof}

\section{Pointed Non-anomalous Semigroups}
\subsection{Morphisms}
In the last section we showed that to understand the universal role of the real numbers, we should try to understand non-anomalous semigroups. We also proved that these semigroups are commutative. In light of this we shall adopt additive notation.

\begin{defn} A positive morphism $f$ between two totally ordered semigroups is a group homomorphism that is an order embedding:
$$f(x+y) = f(x)+f(y), \ \ \ x<y\implies f(x)<f(y). $$
A negative morphism is a group homomorphism that reverses the order
$$x<y\implies f(x)>f(y).$$
\end{defn}

Morphisms are automatically injective, since if $f(x) = f(y)$ then this implies that neither $x>y$ nor $y>x.$ We will denote the category of non-anomalous semigroups by \bo{NAS}, with the arrows being morphisms (both positive and negative) between semigroups. 

\begin{thm}\label{T31} A morphism between two non-anomalous semigroups is determined by where it maps a single non-identity.
\end{thm}
\begin{proof} We proceed by contradiction. Assume that $f_1$ and $f_2$ are morphisms from $A$ to $B$ with $f_1(a) = f_2(a) = b$ but $f_1(x)>f_2(x).$

Assume first that $f_1$ and $f_2$ are positive morphisms. Then $a$ and $b$ have the same sign. We can take $a$ and $b$ to be positive, as otherwise we can study the maps $\dk_B\circ f_{1,2}\circ\dk_A$ between $\bar A$ and $\bar B$.

Using Lemma \ref{L23}, there is an $l\in\bb N$ so that $y = x+la$ is positive. Since $f_1(y)>f_2(y)$ and since there are no anomalous pairs in $B$, there is an $n\in\bb N$ with
$$nf_1(y)>(n+1)f_2(y).$$
By Proposition \ref{P22}, $B$ is Archimedean. So there is a $k\in\bb N$ satisfying $kf_2(y)>b$, and hence
$$knf_1(y) > k(n+1)f_2(y) = knf_2(y) + kf_2(y) > knf_2(y) + b.$$
Finally, by Lemma \ref{L22} there exists an $m\in \bb N$ so that $(m+1)b>knf_1(y)\geq mb.$ Then as $f_1$ and $f_2$ are order preserving,
$$f_1(kny) = knf_1(y)\geq mb =  mf_1(a) = f_1(ma) \implies kny\geq  ma$$
$$(m+1)f_2(a) = (m+1)b>knf_1(y)>knf_2(y)+b = knf_2(y)+f_2(a)$$
$$\implies f_2(ma) = mf_2(a) > knf_2(y) = f_2(kny) \implies ma>kny.$$
This gives us the contradiction we sought.

If instead $f_1$ and $f_2$ are negative morphisms, then we know that $\dk_B\circ f_1$ and $\dk_B\circ f_2$ will be positive morphisms from $A$ to $\bar B$. The above result then implies that $\dk_B\circ f_1 = \dk_B\circ f_2$, and so as $\dk_B$ is an isomorphism, $f_1 = f_2$. \end{proof}

Theorem 2 suggests that the category $\bo{NAS}$ is not the best category to consider when trying to understand non-anomalous semigroups. We should instead be considering \emph{the category of pointed non-anomalous semigroups} $\nas$. In this category, the objects are pairs $(A,a)$ where $A$ is a non-anomalous semigroup and $a\in A$ is not an identity. An arrow $f:(A,a)\ra (B,b)$ is a morphism from $A$ to $B$ with $f(a) = b$. Note that we can always take the basepoint to be positive, since any $(A,a)$ is isomorphic to its dual via the morphism $\dk_A$.

Theorem 2 then says that $\nas$ is a \emph{thin} category, that is, a category where there is at most one morphism between any two objects. Thin categories are much simpler to work with then general categories. For instance, since any arrow between two objects is unique, we can often drop the labels of arrows in a diagram. In proposition \ref{P31} we list a few elementary properties of these categories, which will be useful in this section and the next.

\begin{prop}\label{P31} Let $\fcy C$ and $\fcy D$ be thin categories, and let $F,G:\fcy C\ra\fcy D$ be functors. Then
\begin{enumerate}
\item Any diagram in $\fcy C$ automatically commutes.
\item For objects $X$ and $Y$ in $\fcy C$, if there exists morphisms $X\ra Y$ and $Y\ra X$, then $X\approx Y.$
\item If for all $X\in\text{ob}(\fcy C)$ there is a $\eta_X :FX\ra GX$, then $\eta$ is a natural transformation.
\item If for all $X\in\text{ob}(\fcy C)$, $FX\approx GX$, then $F$ and $G$ are naturally isomorphic.
\item If there exists natural transformations $\eta:1_{\fcy C}\ra GF$ and $\ek: FG\ra 1_{\fcy D}$ then $F$ is left adjoint to $G.$
\end{enumerate}
\end{prop}
\begin{proof} Start with (1). In a diagram, if there are morphisms $f_1f_2...f_n$ and $g_1g_2...g_m$ between objects $X$ and $Y$, then since $\fcy C$ is thin, 
$$f_1f_2...f_n = g_1g_2...g_m$$
and so the diagram commutes. Propositions (2) and (3) are just specific applications of this to the diagrams
\[\begin{tikzcd}
X \arrow{dr}{1_X}\arrow{r}{f}  & Y\arrow{d}{g} & & FX \arrow{d}{Ff}\arrow{r}{\eta_X} & GX \arrow{d}{Gf} \\
 & X                  & & FY               \arrow{r}{\eta_X} & GY
\end{tikzcd}\]
and (4) follows directly from (3). Proposition (5) follows from applying (1) to the unit-counit equations:
\[\begin{tikzcd}
FX \arrow{dr}{1_{FX}}\arrow{r}{F\eta_X} & FGFX\arrow{d}{\ek_{FX}} & & GY \arrow{dr}{1_{GY}}\arrow{r}{\eta_{GY}} & GFGY\arrow{d}{G\ek_Y} \\
& FX                                              &  & & GY  \end{tikzcd}\]
\end{proof}

\subsection{A Lemma}
Showing that $\nas$ is thin has greatly simplified our understanding of the category. We would like to prove a number of other properties of $\nas$, in particular, the existence of a terminal object. This will require a technical lemma:

\begin{lemma}\label{L31} For any family $(A_i,a_i)$ of objects in $\nas$ indexed by a set $I$, there exists an object $(U,u)$ so that there are morphisms $f_i:(A_i,a_i)\ra (U,u).$
\end{lemma}

Proving the above lemma will require some effort, and will occupy us for the rest of this section. We start with a definition. A weakly ordered semigroup $W$ is a set with a relation $\prec $ and an operation satisfying

\begin{quote}
(A1.)\ For every $x,y,z\in W$, if $x\prec y$ and $y\prec z$ then $x\prec z$.
\end{quote}

\begin{quote}
(A2.)\ For every $x,y\in W$, if $x\prec y$ then not $y\prec x$.
\end{quote}

\begin{quote}
(A3.)\ For every $x,y,z\in W$, if $x$ and $y$ are incomparable (so that neither $x\prec y$ nor $y\prec x$) and if $y$ and $z$ are incomparable, then $x$ and $z$ are incomparable.
\end{quote}

\begin{quote}
(A3.)\ For every $x,y,z\in W$, if $(xy)z = x(yz)$
\end{quote}

\begin{quote}
(A5.)\ For every $x,y,z\in W$, $x\prec y \iff xz\prec yz \iff zx\prec zy$.
\end{quote}

These axioms are a generalization of Definition 2.1. Many previous definition, such as non-anomalous and morphism, can be extended to the case of weak orders without change.

\begin{prop}\label{H0} Let $W$ be a weakly ordered semigroup which is non-anomalous. Then there exists a non-anomalous totally ordered semigroup $V$ and a morphism $q:W\ra V.$
\end{prop}
\begin{proof} Define the relation $x\sim y$ on $W$ if $x$ and $y$ are incomparable. From the axioms of a weakly ordered semigroup, it is straightforward to prove that for all $x,x',y\in W$,
$$x\prec y\text{ and } x\sim x' \implies x'\prec y, \ \ \ \  x\sim x' \implies xy\sim x'y $$
Now define $V = W/\sim$ and let $q$ be the quotient map. From the above two statements, it is clear that
$$[x][y] = [xy], \ \ \ [x]\prec [y] \text{ if } x\prec y$$
are independent the representative chosen, and that $V$ is a totally ordered semigroup. This implies that $q$ is a morphism.

We now prove that $V$ is non-anomalous. If $x,y\in W^+$ with $x\prec y$, then since $W$ is non anomalous there exists some $n\in \bb N$ so that $(n+1)x\prec n y$ and so $(n+1)[x]\prec n[y]$. Hence there are no positive anomalous pairs in $V$. An analogous argument holds for negative anomalous pairs.
\end{proof}

In light of Proposition \ref{H0}, to prove Lemma \ref{L31} we need simply to embed each family $(A_i,a_i)$ into a weakly ordered semigroup.

Given some $(A,a)$ in $\nas$, take $a$ to be positive and define a function $\bk_n: A\ra \bb Z$ by
$$\bk_n(x) = \begin{cases} 
     \text{sup}\{m\in\bb N\ |\ ma \leq 2^nx\}  & x\in A^+ \\
    -\text{sup}\{m\in\bb N\ |\ ma + 2^nx \in A^+\}   & \text{otherwise} \\
\end{cases} $$
This is well defined when $x\in A^+$ because $A$ is Archimedean. Since $A$ is non-anomalous, Lemma \ref{L33} then guarantees that it well defined when $x\not\in A^+.$

\begin{lemma}\label{L30} Given $(A,a)$ in $\nas$ and $\bk_n$ as defined above, we have
\begin{enumerate}
\item $\bk_n(a) = 2^n$
\item $2^k\bk_n(x)\leq \bk_{n+k}(x)< 2^k+2^k\bk_n(x).$
\item $\bk_n(x)+\bk_n(y)\leq\bk_n(x+y) \leq 1+\bk_n(x)+\bk_n(y).$
\item If $y < x$ then there exists an $n$ such that $\bk_n(y)+1<\bk_n(x).$
\end{enumerate}
\end{lemma}
\begin{proof}
From the definition of $\bk_n$ the first part of the lemma follows trivially:
$$\bk_n(a) = 2^n.$$
Furthermore notice that
$$\bk_n(x+a) = \bk_n(x)+2^n.$$
Lemma \ref{L33} states for any $w,z\in A$, $z+ka$ and $w+ka$ are in $A^+$ for some $k\in\bb N$. As a consequence, if we can prove properties 2 through 4 for positive $x$ and $y$, then it will follow that they hold for all $x$ and $y.$

So take $x\in A^+$. Then $\bk_n(x)a\leq 2^nx < (\bk_n(x)+1)a$ and so 
$$2^k\bk_n(x)a\leq 2^{n+k}x <2^k(\bk_n(x)+1)a.$$
From the definition of $\bk_n$ this implies that
$$2^k\bk_n(x)\leq \bk_{n+k}(x)< 2^k+2^k\bk_n(x).$$

Now take $x,y\in A^+$. Then as $\bk_n(x)a\leq 2^nx$ and $\bk_n(y)a\leq 2^ny$, we see that
$$(\bk_n(x)+\bk_n(y))a\leq 2^n(x+y)$$
and so $\bk_n(x)+\bk_n(y)\leq \bk_n(x+y).$ Conversely, since $(\bk_n(x)+1)a> 2^n x$ and $(\bk_n(y)+1)a> 2^n y $, we find that
$$(\bk_n(x)+\bk_n(y)+2)a > 2^n(x+y)$$
and so $\bk_n(x)+\bk_n(y)+ 2 > \bk_n(x+y).$ This proves the third part of the lemma.

To prove the fourth property, take $x,y\in A^+$ with $y<x$. Since $A$ is Archimedean there exists an $2^k$ so that $a<2^ky.$ Furthermore, since $2^kx$ and $2^ky$ are not an anomalous pair, there exists some $2^l$ so that
$$2^{k+l}y+a<(2^l+1)(2^ky)<2^{k+l}x$$
It hence follows that $\bk_{k+l}(y) +1\leq \bk_{k+l}(x)$, and this completes the proof. \end{proof}

Given a family of objects $(A_i,a_i)$ in $\nas$, take the coproduct of $A_i$ as abelian semigroups. So we have $B = \bigoplus_{i\in I} A_i$ along with injective homomorphisms $p_i:A_i\ra B.$ Elements of $B$ are just finite formal sums of elements in $A_i$. Hence for each $x\in B$ there exists a unique finite set $I_x$ so that for each $i\in I$, $I_x\cap p_i(A_i)$ has at most one element, and with
$$x = \sum_{p_i(x_i)\in I_x} p_i(x_i).$$
We will denote $d(x) = \text{card}(I_x)$. Furthermore, define $\gk_n:B\ra \bb Z$ by
$$\gk_n(x) = \sum_{p_i(x_i)\in I_x} \bk^{i}_n(x_i).$$
Using part 2 of Lemma \ref{L30} we find that
\begin{equation}\label{l1}2^k\gk_{n}(x)\leq \gk_{n+k}(x) < 2^k\gk_{n}(x) + 2^kd(x).\end{equation}
We now define a relation on $B$, writing $x\prec y$ if there exists an $n\in\bb N$ so that
\begin{equation}\label{l3}\gk_n(x)+d(x)+1\leq \gk_n(y).\end{equation}
Furthermore, write $x\sim y$ if $x$ and $y$ are incomparable.

\begin{lemma}\label{H1} The relation $x\prec y$ holds if and only if for any $M\in\bb N$ there exists an $N\in\bb N$ so that for all $n>N$, $\gk_{n}(x)+M<\gk_{n}(y).$\end{lemma}
\begin{proof} To prove the forward direction, if \eqref{l3} holds for some $n$, then applying \eqref{l1},
$$\gk_{n+k}(x)+2^k\leq 2^k\gk_n(x)+2^kd(x)+2^k\leq 2^k\gk_{n+k}(y)\leq\gk_{n+k}(y).$$
As $2^k$ grows without bound, this implies that for any $M\in\bb N$ there exists an $N\in\bb N$ so that for all $n>N$,
\begin{equation}\label{l4}\gk_{n}(x)+M<\gk_{n}(y).\end{equation}
The backward direction follows from substituting $M = d(x)+1$ into the above equation.
\end{proof}

\begin{lemma}\label{H2} For $x,y\in B,$ $x\prec y\iff x+z\prec y+z$.\end{lemma}
\begin{proof} Using part 3 of Lemma \ref{L30}, we find that
\begin{equation}\label{l2}\gk_n(x)+\gk_n(y)\leq\gk_n(x+y) = \sum_{p_i(w)\in I_{x+y}}\bk_n^i(w) \leq \gk_n(x)+\gk_n(y)+d(x)\end{equation}
since there are at most $d(x)$ elements of $I_{x+y}$ which are the sum of an element in $I_x$ and $I_y.$ If $x\prec y$, from Lemma \ref{H1} we know there exists an $n$ such that
$$\gk_n(x)+2d(x)<\gk_n(y).$$
Now applying \eqref{l2},
$$\gk_n(x+z)+d(x)\leq \gk_n(x)+\gk_n(z)+2d(x)<\gk_n(y)+\gk_n(z)\leq \gk_n(y+z),$$
and so $x+z\prec y+z.$

Conversely, if $x+z\prec y+z$ then from \eqref{l3} there exists an $n$ such that
$$\gk_n(x+z)+d(x)+d(y+z)<\gk_n(y+z).$$
So applying \eqref{l2},
$$\gk_n(x)+\gk_n(z)+d(x)+d(y+z)\leq\gk_n(x+z)+d(x)+d(y+z)$$ $$<\gk_n(y+z) \leq \gk_n(y)+\gk_n(z)+d(y+z).$$
Canceling terms on both sides, this simplifies to $\gk_n(x)+d(x) < \gk_n(y)$
and hence $x\prec y.$
\end{proof}

\begin{lemma}\label{H32} The semigroup $B$ is a weakly ordered semigroup, with $\prec$ as the order. The maps $p_i$ are morphisms with $p_i(a_i) \sim p_j(a_j)$, and $B$ is non-anomalous.
\end{lemma}
\begin{proof}
We start by proving the first part of the lemma. Both axiom 1 in Definition 2.1 and (A1) follows from a straightforward application of Lemma \ref{H1}.

To prove (A2), we need to show that $\sim$ is transitive. If $x\sim y$ then this requires that for every $n,$ \eqref{l3} does not hold, and so 
$$\gk_n(x)+d(x)>\gk_n(y), \ \ \ \gk_n(y)+d(y)>\gk_n(x).$$
Combining these inequalities, we find that if $x\sim y$, then for every $n\in\bb N$,
\begin{equation}\label{l5}\gk_n(y)<\gk_n(x)+d(y)+d(x).\end{equation}
We can now prove that $\sim$ is transitive. Let $x\sim y$ and $y\sim z$, then
$$\gk_n(z) < \gk_n(y)+d(y)+d(z) <\gk_n(x)+d(x)+d(y)+d(z) $$
$$\gk_n(x) < \gk_n(y)+d(y)+d(x) <\gk_n(z)+d(x)+d(y)+d(z).$$
So by Lemma \ref{H1}, this implies that neither $x\prec z$ nor $z\prec x$, so $x\sim z$. Hence (A2) is satisfied.

By definition $B$ is an abelian semigroup, and in particular satisfies axiom 3 in Definition 2.1. Using Lemma \ref{H2}, we then find that axiom 4 is satisfied, so $B$ is a weakly ordered abelian semigroup.
\\

Given $b\in A_i$ we find that $\gk_n(p_i(b)) = \bk_n^i(b)$. Since $d(p_i(b)) = 1$ applying the fourth part of Lemma \ref{L30}, we find that for $b,c\in A_i$, if $b<c$ then $p_i(b)\prec p_i(c).$ So $p_i$ is a morphism from $A_i$ to $B$.

From the first part of Lemma \ref{L30}, $\gk_n(p_i(a_i)) = \bk_n^i(a_i) = 2^n$. So $\gk_n(p_i(a_i)) = \gk_n(p_j(a_j))$ and therefore $p_i(a_i)\not\prec p_j(a_j)$ for all $i,j\in I$. Hence $p_i(a_i) \sim p_j(a_j)$.
\\

Our last task is to show that $B$ is non-anomalous. Take $x,y\in B$ with $x\prec y$. Assume $x$ and $y$ are positive. From Lemma \ref{H1} there exists some $n$ so that
$$\gk_n(x)+d(x)+1 < \gk_n(y).$$
Take some $m\in\bb N^+$ with $m>\gk_n(y)$. Then multiplying the above inequality by $m,$
\begin{equation}\label{l6}m(\gk_n(x)+2d(x))+m<m\gk_n(y) \implies m(\gk_n(x)+2d(x))<(m-1)\gk_n(y).\end{equation}
If we repeatedly apply \eqref{l2} to $mz$ for any $z\in B$, we find that
$$m\gk_n(z) \leq \gk_n(mz) \leq m(\gk_n(z)+d(z)).$$
In particular, applying this to \eqref{l6},
$$\gk_n(mx)+md(x)\leq m(\gk_n(x)+2d(x)) < (m-1)\gk_n(y)\leq\gk_n((m-1)y).$$
Since $d(mx) = d(x)$ and since $m>0$, this implies that $mx\prec (m-1)y$. An analogous argument holds for the case where $x$ and $y$ are negative.
\end{proof}

We are finally in a position to prove Lemma \ref{L31}:

\begin{proof}[Proof (Lemma \ref{L31})] By Lemma \ref{H32}, any family $(A_i,a_i)$ embeds into some non-anomalous weakly ordered semigroup $B$ with morphisms $p_i:A_i\ra B$, and furthermore, $p_i(a_i) \sim p_j(a_j)$. Applying Proposition \ref{H0}, there is a non-anomalous semigroup $U$ and a morphism $q:B\ra U$. So the maps $q\circ p_i$ are morphisms from $A_i$ to $U.$ Furthermore, since $p_i(a_i) \sim p_j(a_j)$ we see that $q\circ p_i(a_i) = q\circ p_j(a_j) = u$ for every $i,j\in I$. This completes the proof.
\end{proof}

\subsection{Bicompleteness of $\nas$}
With the proof of Lemma \ref{L30} complete, we are know free to prove the bicompleteness of $\nas$. This in particular means that $\nas$ has a terminal object, which we shall denote $(\bb R,r).$ We will then show that $\nas$ can be understood entirely in terms of the additive structure on $\bb R$, and as a consequence is essentially small.

\begin{thm}\label{T32} $\nas$ is cocomplete with initial object $(\bb N,1)$.
\end{thm}
\begin{proof} Take a family $(A_i,a_i)$ of objects in $\nas$ indexed by a set $I$. From Lemma \ref{L31} there exists a $(U,u)$ so that there are morphisms $f_i:(A_i,a_i)\ra(U,u)$. Let $U_j$ be the set of subsemigroups of $U$ which contain every $f_i(A_i),$ and define $\overline U = \bigcap_{j}U_j$. Then $\overline U$ is a non-anomalous subsemigroup of $U,$ with the universal property
\[\begin{tikzcd}
(A_i,a_i)\arrow{r}\arrow{rd} & (U_j,u)\arrow{r} & (U,u) \\
 & (\overline U,u)\arrow[dashrightarrow]{u}\arrow{ru} & \\
\end{tikzcd}\]
Let $(V,v)$ be also such that there are morphisms $f_i:(A_i,a_i)\ra(V,v)$ for each $i\in I$. There then exists a $(\overline V,v)$ satisfying the above universal property for $V$. From Lemma \ref{L31}, there is some $(W,w)$ so that there exists morphisms $g_U:(U,u)\ra(W,w)$ and $g_V:(V,v)\ra(W,w)$. We also now have a $(\overline W,w)$ for $W.$
\[\begin{tikzcd}
 & (\overline U,u) \arrow{r}\arrow[dashrightarrow]{rd}  & (U,u)\arrow{rd} & \\
(A_i,a_i)\arrow{ru}\arrow{rd}\arrow{rr} & & (\overline W,w)\arrow{u}\arrow{d}\arrow[dashrightarrow]{ld}\arrow[dashrightarrow]{lu}\arrow{r}  & (W,w)\\
 & (\overline V,v) \arrow{r}\arrow[dashrightarrow]{ru}  & (V,v)\arrow{ru} & 
\end{tikzcd}\]
From the universal property of $(\overline W,w)$, we know that there is a map from $(\overline W,w)\ra(\overline U,u).$ But then there exists a map $(\overline W,w)\ra(U,u)$ and so by the universal property of $(\overline U,u)$ there is a map $(\overline U,u)\ra (\overline W,w).$ Since $\nas$ is thin, $(\overline U,u)\approx (\overline W,w)$. By symmetry, $(\overline V,v)\approx (\overline W,w)$ and so $(\overline V,v)\approx (\overline U,u).$ So fixing some $(U,u)$, we see that $(\overline U,u)$ is the coproduct of $(A_i,a_i)$:
\[\begin{tikzcd}
(A_i,a_i)\arrow{r}\arrow{rd} & (\overline U,u)\approx (\overline V,v) \arrow[dashrightarrow]{d} \\
 & (V,v)
\end{tikzcd}\]
Because $\nas$ is thin, all diagrams automatically commute. Since $\nas$ has arbitrary coproducts, any non-empty small diagram has a colimit. 

All that remains now is to prove $\nas$ has an initial object. Let $(A,a)\in\text{ob}(\nas)$. We can then define the map $f:(\bb N,1)\ra(A,a)$ by $f(n) = n a$. Since $\nas$ is thin $f$ is unique and hence $(\bb N,1)$ is the initial object.\end{proof}

We will denote the coproduct on $\nas$ by $\oplus$. Notice that $\oplus$ is idempotent: $$(A,a)\oplus(A,a)\approx (A,a)$$
This is a general property of thin categories, following from the fact that in the diagram
\[\begin{tikzcd}
(A,a)\arrow{r}\arrow{rd}{f} & (A,a)\arrow[dashrightarrow]{d} & (A,a)\arrow{l}\\
& (B,b) \arrow[leftarrow]{ru}{g}&
\end{tikzcd}\]
the morphisms $f$ and $g$ must be equal.

We will now show that $\nas$ is a complete category, and in particular has a terminal object. This requires a couple of lemmas. These act to constrain the size of $\nas.$

We will say that an object $(A,a)$ is \emph{elementary} if $A$ is generated by $a$ and some other element $b$; that is $A = \{ma+nb\ |\ m,n\in\bb N\}.$ In other words, $A$ is a rank-2 semigroup pointed by one of its generators.

\begin{lemma}\label{L32} The class of isomorphism classes of elementary semigroups form a set.\end{lemma}
\begin{proof} 
Take some elementary semigroup $(A,a)$ with $a$ to be positive and choose some $b\in A$ so that $a$ and $b$ generate $A$. We define the function $f_{A,a,b}:\bb N^4\ra \bb N$ by
$$f_{A,a,b}(m,n,p,q) = \begin{cases} 
      2 & \text{ if } ma+nb>pa+qb \\
      1 & \text{ if } ma+nb=pa+qb\\
      0 & \text{ if } ma+nb<pa+qb 
   \end{cases}.$$
Given some other elementary semigroup $(B,\ak)$ generated by $\ak$ and $\bk$, if $f_{B,\ak,\bk} = f_{A,a,b}$ then it is manifest that the function $g:A\ra B$ via
$$g(ma+nb) = m\ak+n\bk $$
is an isomorphism. Since the function from $\bb N^4\ra \bb N$ form a set, this implies that the class of isomorphism classes of elementary semigroups form a set.
\end{proof}

\begin{lemma}\label{L33} Any object $(A,a)$ in $\nas$ is isomorphic to the coproduct of elementary objects.\end{lemma}
\begin{proof} For each $b\in A$ there is an elementary non-anomalous subsemigroup of $A$ generated by $a$ and $b.$ Denote this object by $(E_b,a).$ Since $A$ is a set we can now take the coproduct of each $(E_b,a)$,
$$(C,c) = \bigoplus_{b\in B} (E_b,a).$$
From the universal property of the coproduct, we have the following diagram
\[\begin{tikzcd}
(E_b,a)\arrow{rd}{g_d}\arrow{dd}{f_b} \\
& (C,c)\arrow[dashrightarrow]{ld}{h} \\
(A,a) \\
\end{tikzcd}\]
Since for any $b\in A$ there exists an $E_b$ so $b$ is in the image of $f_b$, this means that $b$ is in the image of $h$ and hence $h$ is surjective. Morphisms are automatically injective, so $h$ is an isomorphism.
\end{proof}

\begin{thm}\label{main} The category $\nas$ is complete, and in particular has a terminal object denoted by $(\bb R,r).$
\end{thm}
\begin{proof} Let $I$ be the set of isomorphism classes of elementary semigroups, which we know exists because of Lemma \ref{L32}. Using Theorem \ref{T32} we can take the coproduct of all elementary semigroups, 
$$(\bb R,r) \approx \bigoplus_{(E_i,e_i)\in I}(E_i,e_i).$$
By Lemma \ref{L33} any object $(B,b)$ is isomorphic to the coproduct of elementary objects. Using the idempotency of the coproduct,
$$(B,b)\oplus \bigoplus_{(E_i,e_i)\in I}(E_i,e_i)\approx \bigoplus_{(E_i,e_i)\in I}(E_i,e_i)\approx (\bb R,r) .$$
Therefore from the universal property of the coproduct there is a morphism $(B,b)\ra(\bb R,r)$. This is unique since $\nas$ is thin, and hence $(\bb R,r)$ is the terminal object in $\nas$.

We now show that $\nas$ has arbitrary products. Take a collection of objects $(A_j,a_j)$ indexed by a set $J$. Because $(\bb R,r)$ is a terminal object, there exists morphisms $f_j:(A_j,a_j)\ra (\bb R,r)$. Let us now defines
$$B = \bigcap_{j\in J} f_j(A_j).$$
This is non-empty since it contains $r$. For any $x,y\in B$, $x,y\in f_j(A_j)$ for every $j\in J$. It then follows that $x+y\in f_j(A_j)$ and so $x+y\in B$. Therefore $B$ is a subsemigroup of $\bb R$. It inherits the order on $\bb R$ and is so a non-anomalous semigroup.

\[\begin{tikzcd}
(C,c)\arrow{d}{g_j}\arrow[dashrightarrow]{r} & (B,r)\arrow{ld}\arrow{d}\\
(A_i,a_i)\arrow{r}{f_j} & (\bb R,r)
\end{tikzcd}\]

From the definition of $B$ there are inclusion morphisms $(B,r)\ra(A_j,a_j)$. Let there be morphisms $g_j:(C,c)\ra(A_j,a_j).$ Then each $g_j\circ f_j$ is a morphism $(C,c)\ra(\bb R,r)$ and hence all of these morphisms are equal. So $g_j\circ f_j(C)\subset B$ and hence there is a map from $(C,c)\ra(B,r)$. So $(B,r)$ satisfies the universal property of the product.

Since $\nas$ has arbitrary products and is thin, it follows that $\nas$ has all non-empty limits. We know that $\nas$ also has a terminal object, so $\nas$ is complete.
\end{proof}

To finish this section, we will show how $\nas$ can be reconstructed from the additive structure of $\bb R$. Choose some basepoint $r\in\bb R$. Define $\bo{sub}_{\bb R}$ to be the thin category where the objects are subsemigroups of $\bb R$ containing $r$, and where there is a morphism between $S_1$ and $S_2$ iff $S_1\subset S_2$. This is a small and skeletal category.

\begin{prop}\label{P32} There is an equivalence of categories between $\nas$ and $\bo{sub}_{\bb R}$. In particular, $\nas$ is essentially small.
\end{prop}
\begin{proof}
By Theorem \ref{main}, every object $(A,a)$ in $\nas$ uniquely embeds into $(\bb R,r)$ via some map $f_A$. So we can define a functor $M:\nas\ra\bo{sub}_{\bb R}$ which takes $(A,a)$ and maps it to $f_A(A)\subset \bb R.$ If there is a map $g:(A,a)\ra(B,b)$ then $f_A = f_B\circ g$ and hence $f_A(A)\subset f_B(B)$. So we define $M$ to take the morphism $g:(A,a)\ra(B,b)$ to the morphism $f_A(A)\ra f_B(B)$.

Any subsemigroup of $\bb R$ inherits a total order from $\bb R$, and so can be made into a non-anomalous semigroup. Define the functor $N:\bo{sub}_{\bb R}\ra\nas$ which takes the subsemigroup $S_1\subset \bb R$ and maps it to $(S_1,r)$. If $S_1\subset S_2$, then the injection map $i:S_1\ra S_2$ is a morphism which takes $r$ to $r$, and so we define $N(i):(S_1,r)\ra(S_2,r)$.

It is now manifest that for every $(A,a)\in\text{ob}(\nas)$, $NM(A,a)\approx (A,a)$. It is similarly straightforward to see that $MN(S) = S$ for every $S\in\text{ob}(\bo{sub}_{\bb R})$. So from Proposition \ref{P31}, there are natural isomorphisms $NM \approx 1_{\nas}$ and $MN\approx 1_{\bo{sub}_{\bb R}}$, so $M$ and $N$ are part of an adjoint equivalence. Since $\bo{sub}_{\bb R}$ is small, this implies that $\nas$ is essential small.
\end{proof}

\section{Groups and Orders}
\subsection{Archimedean Groups}
In the previous section we defined $\bb R$ to be the terminal object in $\nas$. Whilst this is a philosophically appealing definition, we have not proved any special properties about $\bb R$. In this section, we shall rectify this by providing a unique characterization of the group and order structure on $\bb R$.

We will begin by discussing non-anomalous groups. For these objects, the converse of Proposition \ref{L22} is true.

\begin{prop} \label{P91} All Archimedean groups are non-anomalous. \end{prop}
\begin{proof} Let $x,y$ be an anomalous pair in a group $G$ with $x>y$. We can take $x,y\in G^+$, since if they are not we can use $-x$ and $-y$ instead. Since we know that for every $n\in\bb N$, $(n+1)y>nx $, this means that $y>n(x-y)$ for any $n$. There $G$ cannot be Archimedean. So any Archimedean group must be non-anomalous. \end{proof}

As a result, non-anomalous groups are usually called Archimedean groups; likewise, for Archimedean rings and fields. We will write $\bo{AG}$ for the category of Archimedean groups. The arrows in this category are group homomorphisms which preserve or reverse the order, so that the category is a full subcategory of $\bo{NAS}.$

Analogous to our construction of $\nas$, we can consider $\bo{AG}_\bullet$, the category of pointed Archimedean groups. This is a full subcategory of $\nas$, and is hence thin category. It is straightforward to prove that $(\bb Z,1)$ is an initial object in $\bo{AG}_\bullet.$

\subsection{Order Properties of $\bb R$}
We will provide a unique characterization of $\bb R$ using order-theoretic properties. In particular we show that $\bb R$ is a group. Our characterization is due to Loonstra \cite{13}, though as our definition of the reals is different to the one used by Loonstra, our proof is different.

If an order contains no maximum or minimum element we say that it is \emph{unbound}. An order is \emph{dense} if for every $x>y$ there exists a $z$ so that $x>z>y.$ This just says that between any two elements is a third element.

For any order $A$, a subset $U$ is bound if there exists an $a\in A$ so that for every $u\in U$, $u\leq a.$ If for every bound subset $U$ in $A$ there is a least upper bound, we say that $A$ is \emph{complete}.

\begin{prop}\label{P92} Any complete totally ordered group is Archimedean.
\end{prop}
\begin{proof}
We will use proof by contradiction. Assume $G$ is complete, and that $x,y\in G$ satisfy $ny<x$ for all $n\in\bb N$. Since $Y = \{ny\ |\ n\in\bb N\}$ is a bound set, it has a least upper bound $z$. But then for every $n\in \bb N$,
$$(n+1)y < z \implies ny<z-y$$
and so $z-y$ is a bound on $Y$ which is smaller than $z$. We have a contradiction.
\end{proof}

\begin{prop} There exists an totally ordered group which is dense and complete.
\end{prop}
\begin{proof}
Any construction of the reals can be simplified into the construction of a dense and complete totally ordered group. Since reproducing such a construction here would be unwieldy and not particularly enlightening, we will not provide a complete proof. We will provide a sketch of a particularly simple construction. The abelian group
$$T = \langle x_1,x_2,...| 2x_{i+1} = x_i\rangle \text{ with order } x_i>0$$
is a dense group. We can now Dedekind complete $T$ to obtain a dense and complete totally ordered group. \end{proof}

\begin{lemma}\label{L91} In a dense totally ordered group $T$, for each $b\in T^+$ and $n\in\bb N$ there exists some $c\in T^+$ such that $b>nc$.\end{lemma}
\begin{proof} Let $b>0$. By denseness there exists a $b'$ such that $b>b'>0.$ If both $b<2(b-b')$ and $b<2b'$, then $2b<2b$ and we would have a contradiction. So either $b\geq 2(b-b')$ or $b\geq 2b'$. Hence for every $b>0$ there exists a $c>0$ so that $b\geq 2c$. By induction, for any $n\in\bb N$ and any $b>0$ there is a $c>0$ such that $b\geq 2^nc$. Since $2^nc>nc$, we find $b\geq2^n c>nc$ and this completes the proof.
\end{proof}

\begin{thm}\label{T91} Every dense and complete totally ordered group is isomorphic to $\bb R$.\end{thm}
\begin{proof} Let $T$ be a dense and complete totally ordered group. Using Proposition \ref{P92} and the definition of $\bb R$, there is a morphism $f:(T,t)\ra(\bb R,r).$ We will take both $t$ and $r$ to be positive. 

Let $q\in \bb R^+$ and define the set $L = \{a\in T:f(a)<q\}$. Since $\bb R$ is Archimedean, there is a $k\in\bb N$ so that $kr>q$, and hence $kt$ is an upper bound on $L$. Furthermore, $0\in L$, so $L$ has a least upper bound $l.$ Then for every $b\in\bb R^+$,
$$f(l+b)\geq q \implies f(b)\geq q - f(l) $$
$$f(l-b)\leq q \implies f(b)\geq f(l) - q $$
Fix some $b\in\ T^+$. By Lemma 7, for each $n\in\bb N$ there is a $c\in T^+$ so that $b>nc$. Then for every $n\in \bb N$
$$f(c)\geq q - f(l) \implies f(b)\geq n(q - f(l))$$
$$f(c)\geq q + f(l) \implies f(b)\geq n(f(l)-q)$$
Since $\bb R$ is Archimedean, we hence find that 
$$q - f(l)\leq 0 \implies q\leq f(l)$$
$$f(l) - q\leq 0 \implies q\geq f(l)$$
and so $f(l) = q.$ So every positive $q\in \bb R$ is in the image of $f$. But if $q$ is negative, $-q$ is in the image of $f$ and hence so is $q$. Therefore $f$ is surjective and so is an isomorphism.
\end{proof}
Theorem \ref{T91} implies that $\bb R$ is a group. Since $\bo{AG}_\bullet$ is a subcategory of $\nas$, $(\bb R,r)$ is the terminal object in $\bo{AG}_\bullet$.

\section{Rings and Fields}
\subsection{Archimedean Rings}
We will now consider the relationship between Archimedean groups, and Archimedean rings and fields. Specifically we shall show that the reals have a naturally field structure. When combined with \ref{T91}, we will prove that $\bb R$ is the unique ordered field up to a unique isomorphism.

We shall write the product of two elements $a$ and $b$ in a ring as $ab$, and the multiplicative identity as $1.$ An Archimedean ring is a ring over an Archimedean group, with the additional axiom that if $a>0$ and $b>0$, then $ab > 0$. We shall also demand that $1\neq 0$ so that the zero-ring is not Archimedean. Multiplication in an Archimedean ring is automatically commutative, which shall be proved shortly.

Our results can easily be extended to Archimedean semirings and semifields (that is, dropping the additive identity and inverse axioms, e.g. $\bb R^+$), but we shall restrict to rings and fields since these are the objects traditionally studied.

In this section, we aim to show there is an adjoint equivalence $\Gk,\Fk$ between $\bo{AR}$ and a full subcategory $\bo{AG}_\bullet^I$ of $\bo{AG}_\bullet.$ Our first step is to show that the faithful forgetful functor $F:\bo{AR}\ra\bo{AG}$ factorizes through $\bo{AG}_\bullet$. Any function between two rings must preserve the multiplicative identity. Since $1\neq 0$, we see that $F$ factorizes into a functor from $H:\bo{AR}\ra\bo{AG}_\bullet$ composed with the forgetful functor $G:\bo{AG}_\bullet\ra \bo{AG}.$ The functor $H$ takes a ring $R$ and maps it to its underlying Archimedean group, pointed by the multiplicative identity; we can write this as $F(R) \approx (R,1).$ This functor must be faithful, and so we can deduce that $\bo{AR}$ is a thin category.
\[\begin{tikzcd}
\bo{AR}\arrow[xshift=0.75ex]{r}{\Gk}\arrow[xshift=-0.75ex]{rd}{H}\arrow{d}{F}  & \arrow[xshift=0.5ex]{l}{\Fk} \bo{AG}_\bullet^I \arrow{d}{I}\\
\bo{AG}\arrow{r}{G}  & \bo{AG}_\bullet \\ 
\end{tikzcd}\]

Our next step is to define $\bo{AG}_\bullet^I$ and factorize $H = I\Gk$. We start with a lemma.

\begin{lemma}\label{L41} In an Archimedean ring $A$, every non-zero element $a\in A$ defines a morphism of the underlying Archimedean groups via $r_a(b) = ab.$
\end{lemma}
\begin{proof} 
By the distributive law, $r_a$ is a group homomorphism. Furthermore, if $a$ is positive, then this preserves the order. To prove this, take $b>c$, which implies $b-c>0$. Then $a(b-c)> 0$ and so
$$r_a(b-c) = a(b-c)>0 \implies r_a(b)>r_a(c).$$
Therefore $r_a$ is a positive morphism if $a>0.$ If $a$ is negative, then
$$-r_a(b) = -ab = (-a)b = r_{-a}(b)$$
and hence $r_a$ is a negative morphism. This completes the proof.
\end{proof}

For an Archimedean group $A$, an \emph{initial element} $i\in A$ is an element such that there is a morphism $(A,i)\ra(A,a)$ for all $a\in A.$ 

\begin{prop}The identity in an Archimedean ring is an initial element of the underlying Archimedean group.
\end{prop}
\begin{proof} Let $A$ be a ring and $a\in A$ any non-zero element. By Lemma \ref{L41}, any non-zero element $a$ defines a morphism $r_a$ by $r_a(b) = ab$. Since $r_a(1) = a$, there is a morphism $(A,1)\ra(A,a)$ and so $1$ is an initial object.
\end{proof}

We define $\bo{AG}_\bullet^I$ as the category of Archimedean group pointed by an initial element. From the above proposition we can factorize $H = \Gk I$. Here $\Gk:\bo{AR}\ra\bo{AG}_\bullet^I$ takes the underlying group of a ring and points by the identity, $\Gk(R) = (R,1)$, and $I$ is the inclusion functor $\bo{AG}_\bullet^I\ra \bo{AG}_\bullet.$ Our final task to show $\Gk$ is part of an adjoint equivalence. First we prove another lemma, from which the commutativity of Archimedean rings will follow as a consequence.

\begin{lemma}\label{L42} Given Archimedean rings $R_1$ and $R_2$, if $\Gk R_1\approx \Gk R_2$ then $R_1\approx R_2$.
\end{lemma}
\begin{proof}
Given that $\Gk R_1\approx \Gk R_2$, there is an isomorphisms $i_1: GI\Gk(R_1)\ra GI\Gk(R_2)$. Using Lemma \ref{L41}, define morphisms $r^1_a$ and $r^2_a$ on $GI\Gk(R_2)$ via
$$r^1_a(b) = i_1(i_1^{-1}(a)i_1^{-1}(b)), \ \ \ r^2_a(b) = ab.$$
Since $HR_1\approx HR_2$, we know that $i_1(1) = 1$, and so 
$$r^1_a(i_1(1)) = i_1(i_1^{-1}(a)i_1^{-1}(i_1(1))) =  i_1(i_1^{-1}(a)) = a = r^2_a(1).$$
By Theorem \ref{T31} $r^1_a$ and $r^2_a$ are equal. So for every $c,d\in R_1$
$$i_1(cd) = r^1_{i_1(c)}(i_1(d)) = r^2_{i_1(c)}(i_1(d)) = i_1(c)i_1(d)$$
So $i_1$ is a Archimedean ring isomorphism, and hence $R_1\approx R_2.$
\end{proof}

As a consequence of the previous lemma, we prove that Archimedean rings are necessarily commutative.
\begin{prop} Every Archimedean ring is commutative.
\end{prop}
\begin{proof} Given an Archimedean ring $R$, we can define another Archimedean ring on $\bar R$ where the order of multiplication is reversed. That is, there is a bijection $i:R\ra \bar R$ which is an Archimedean group isomorphism, and
$$i(ab) = i(b)i(a).$$
The multiplicative identity of $R$ is $i(1)$, and so $HR_1\approx HR_2$. By Lemma \ref{L42}, $i$ is a ring isomorphism. But then
$$i(ab) = i(a)i(b) = i(b)i(a) = i(ba) $$
and so $ab = ba.$
\end{proof}

\begin{prop}\label{T42} The functor $\Gk$ is part of an adjoint equivalence between $\bo{AR}$ and $\bo{AG}_\bullet^I$. \end{prop}
\begin{proof} We will begin by showing every object in $\bo{AG}_\bullet^I$ has an associated ring structure. Since for every $a\in A$ there exists a unique $f_a: (A,i)\ra (A,a)$, define the binary operation
$$a\times b = f_a(b)$$
mapping $A\times A\ra A.$
We will prove that this operation is associative, has identity $i,$ obeys the distributive law and interacts properly with the order. We begin with the identity:
$$a\times i = f_a(i) = a, \ \ \ i\times a = f_i(a) = \text{id}(a) = a. $$
Right-distributivity is also easy
$$a\times(b+c) = f_a(b+c) = f_a(b)+f_a(c) = a\times b + a\times c.$$
Left-distributivity is a little trickier; if we define $g(c) = a\times c+ b\times c = f_a(c)+f_b(c)$ then this is a morphism with $g(i) = a+b = f_{a+b}(i).$ Hence for every $c\in A$
$$(a+b)\times c = f_{a+b}(c) = g(c) = a\times c+ b\times c .$$
Now we deduce associativity. Define $p_{a,b}(c) = (a\times b)\times c$ and $q_{a,b}(c) = a\times(b\times c).$ As $p_{a,b}(i) = q_{a,b}(i) = a\times b$, we have for every $c$,
$$(a\times b)\times c = p_{a,b}(i) = q_{a,b}(i) = a\times(b\times c).$$
Finally, we show that if $a,b>0$ then $a\times b>0.$ We have the freedom to choose $i$ to be positive. Then if $a$ and $b$ are positive, $f_a$ and $f_b$ are positive morphisms. So $f_a\circ f_b$ are positive morphisms, and as $i>0$, $f_a(f_b(i)) = a\times b > 0.$
\\

For a given $(A,i)\in\text{ob}(\bo{AG}_\bullet)$, let us write the above ring structure as $\Fk (A,i).$ We will show that $\Fk$ lifts to a functor $\Fk: \bo{AG}_\bullet^I\ra \bo{AR}$. This simply requires us show that any morphism $g:(A,i)\ra(B,j)$ is a ring homomorphism from $\Fk(A,i)$ to $\Fk(B,j).$ For any $a\in A$ we have the commutative diagram
\[\begin{tikzcd}
(A,i) \arrow{r}{f^A_a}\arrow{d}{h} & (A,a) \arrow{d}{h} \\ 
(B,j) \arrow{r}{f^B_{h(a)}} & (B,h(a))\end{tikzcd}\]
and hence $h$ is a ring homomorphism
$$h(a\times b) = h\circ f^A_a(b) = f^B_{h(a)}\circ h(b) = h(a)\times h(b).$$

We will now show that $\Gk$ and $\Fk$ are adjoint equivalences. For any $A\in\text{ob}(\bo{AG}_\bullet^I)$, it is evident that $A\approx \Gk\Fk A$ and so $1_{\bo{AG}_\bullet^I}$ is naturally isomorphic to $\Gk\Fk$. Lemma \ref{L42} then guarantees that $1_{\bo{AR}}$ and $\Fk \Gk$ are naturally isomorphic, and this completes the proof. \end{proof}

The initial object in $\bo{AG}_\bullet$, $(\bb Z,1)$, is in $\bo{AG}_\bullet^I$. This is because for any other $(\bb Z,n)$ is in $\bo{AG}_\bullet$ and hence $(\bb Z,1)\ra(\bb Z,n).$ So as a corollary of Proposition \ref{T42}, $(\bb Z,1)$ is associated with unique Archimedean ring structure $\bb Z\in \text{ob}(\bo{AR})$, and furthermore $\bb Z$ is the initial object in $\bo{AR}.$

\subsection{Archimedean Fields}
Our final topic is to discuss the relationship between Archimedean groups and Archimedean fields. An Archimedean field is an Archimedean ring where every element is invertible. Therefore $\bo{AF}$, the category of Archimedean fields, is a full subcategory of $\bo{AR}.$

Given an Archimedean group $A$, define a \emph{terminal element} $t\in A$ to be an elements such that for every $a\in A$, there exists a morphism $(A,a)\ra(A,t).$ This is dual to our definition of an initial element.

\begin{prop}\label{P41} If $A$ has a terminal element, then every element in $A$ is both initial and terminal.
\end{prop}
\begin{proof} Let $t\in A$ be a terminal element, and take some arbitrary $a\in A$. Then there exists an $f:(A,a)\ra(A,t)$. We also know that there exists a map $g:(A,f(t))\ra (A,t)$, and so $g\circ f : (A,t)\ra (A,t)$. By Theorem \ref{T31}, $g\circ f = 1_A$ and so $f$ is invertible. So $t$ is both an initial and terminal object.

But this means every element is an initial and terminal element, since for any $a$ and $b$ there exists morphisms $(A,a)\ra(A,t)\ra(A,b)$.
\end{proof}

Let $\bo{AG}_\bullet^T$ be the category of Archimedean groups pointed by terminal objects. The above proposition means that $\bo{AG}_\bullet^T$ is a subcategory of $\bo{AG}_\bullet^I$. In particular, $\Fk$ maps $\bo{AG}_\bullet^T$ to a subcategory of $\bo{AR}.$

\begin{prop}\label{T43} There is an equivalence of categories between $\bo{AG}_\bullet^T$ and $\bo{AF}.$
\end{prop}
\begin{proof} Let $R:\bo{AF}\ra\bo{AR}$ and $T:\bo{AG}_\bullet^T\ra \bo{AG}_\bullet^I$ be the inclusion functors. From the previous section, we have an adjoint equivalence $\Gk,\Fk$ between $\bo{AR}$ and $\bo{AG}_\bullet^I.$
\[\begin{tikzcd}
\bo{AR} \arrow[xshift=0.75ex]{r}{\Gk} & \arrow[xshift=0.75ex]{l}{\Fk} \bo{AG}_\bullet^I \\
\bo{AF} \arrow{u}{R}\arrow[xshift=0.75ex]{r}{\Psi} & \arrow[xshift=0.75ex]{l}{\Xi}\bo{AG}_\bullet^T\arrow{u}{T} \end{tikzcd}\]
Let $F$ be a field in $\bo{AF}$, and take some $a\in F$. From Lemma \ref{L41}, we have an Archimedean group morphism $r_{a^{-1}}(b) = a^{-1}b$. Since $r_{a^{-1}}(a) = aa^{-1} = 1$, this is a morphism from $(A,a)\ra(A,1)$. As $a$ was arbitrary, $1$ is a terminal element. So $\Gk R = T\Psi$ for some functor $\Psi:\bo{AF}\ra\bo{AG}_\bullet^T.$

Now let $(A,t)$ be an object in $\bo{AG}_\bullet^T.$ By Proposition \ref{P41}, for non-zero $a\in A$ is terminal so there is an isomorphism $i:(A,t)\ra (A,a)$. Let $\times$ be the product on $\Fk(A,t).$ Using Lemma \ref{L41} to deduce that $r_a(b) = a\times b$ is an morphism of $A$, we know that
$$i(t) = a = a\times t = r_a(t) $$
and so by Theorem \ref{T31}, $i = r_a$. But then
$$t = i(i^{-1}(t)) = a\times i^{-1}(t) $$
and so $a$ has a multiplicative inverse $i^{-1}(t).$ Hence $\Fk(A,t)$ is a field. Since this is the case for any object in $\bo{AG}_\bullet^T$, there exists a $\Xi: \bo{AG}_\bullet^T \ra \bo{AF}$ so that $\Fk T = R\X.$

We now combine the relationships $\Fk T = R\X$ and $\Gk R = T\Psi$, we find that
$$\Gk\Fk T = \Gk R\Xi = T\Yk\Xi $$
$$\Fk\Gk R = \Fk T\Yk = R\Xi\Yk.$$
Since $\Fk$ and $\Gk$ are adjoint equivalences, there is a natural isomorphism from $T\ra T\Yk\Xi$ and $R\ra R\Xi\Yk$. But since $R$ and $T$ are fully faithful functors, this implies that $\Yk$ and $\Xi$ are adjoint equivalences.
\end{proof}

The terminal object $(\bb R,r)$ is pointed by a terminal element, since for any other $q\in\bb R$ there is a morphism $(\bb R,q)\ra(\bb R,r).$ So by Proposition \ref{T43}, $(\bb R,r)$ has an associated field structure. In other words, if we choose some element $1\in\bb R$, then there is a unique choice of product on $\bb R$ which has identity $1$, and this is automatically a field. Furthermore, $\bb R$ is the terminal object in $\bo{AF}.$
\begin{thm} Up to a unique isomorphism, $\bb R$ is the unique complete ordered field.\end{thm}
\begin{proof} In Theorem \ref{T91} we proved that $\bb R$ was the unique complete and dense ordered group. In the above discussion we showed that $\bb R$ has a field structure, which is unique up to a choice of $1\in\bb R$. Since $\bo{AF}$ is a thin category, any isomorphism between two fields is unique. So $\bb R$ is the unique complete and dense ordered field, up to a unique isomorphism.

To complete the proof we shall show that the assumption of denseness is redundant. Given any two $x,y$ in an ordered field $F$ with $x>y$, we find that $2x>x+y>2y$. Multiplying this inequality by $2^{-1}$, we find that
$$x>2^{-1}(x+y) > y $$
and since our choice of $x$ and $y$ was arbitrary, $F$ is dense. This completes the proof.
\end{proof}

Since any non-anomalous semigroup can be embedded into the reals, it is natural to ask how to decide whether two subsemigroups of the reals are equal. This problem can now be solved with the field structure on the reals. The below proposition is a slight generalization of a result due to Hion \cite{41}.

\begin{prop} Given two non-anomalous semigroups $A,B$ with embeddings $i:A\ra\bb R$ and $j:B\ra\bb R$. Furthermore let there exist a morphism $f:A\ra B$. Then there exists an $\lk\in \bb R$ so that $\lk i(a) = f(j(a))$
\end{prop}
\begin{proof} Take some $\ak\in A$ and define $\lk = f(j(\ak))(i(\ak))^{-1}.$ By Lemma \ref{L41}, this defines a morphism $r_\lk(\ak) = \lk \ak$, and furthermore
$$r_\lk(\ak) = \lk i(\ak) = f(j(\ak))(i(\ak))^{-1}i(\ak) = f(j(\ak)).$$
So by Theorem \ref{T31}, $r_\lk \circ i = f\circ j$, and so for every $a\in A$, $\lk i(a) = f(j(a)).$
\end{proof}


\begin{thebibliography}{1}
\bibitem{11} Alimov, N. L. (1950). \emph{On ordered semigroups}, Math. Rev. vol. 12, p. 480.
\bibitem{12} H\"older, O. (1901). \emph{Die Axiome der Quantit\"at und die Lehre vom Mass}, Ber. Verh. S\"achs. Ges. Wiss. Leizig, Math. Phys. Cl. vol. 53.
\bibitem{13} Loonstra, L. (1946). \emph{Ordered groups.} Proc. Nederl. Akad. Wetensch., \bo{49}.
\bibitem{21} Chehata, C. G. (1953). \emph{On an ordered semigroup,} J. London Math. Soc. vol 28.
\bibitem{22} Vinogradov, A. A. (1953). \emph{On the theory of ordered semigroups,} Math. Rev. vol. 17.
\bibitem{23} Clay, A., Rolfsen, D. (2015). \emph{Ordered Groups and Topology}, arXiv:1511.05088v1.
\bibitem{41} Hion, (1954). {\cyr Arhimedovski Uporyadochenn\cprime ie Kol\cprime tsa, Uspekhi Mat. Nayk}, \bo{9} : 4.
\end{thebibliography}
\end{document}